\documentclass[12pt, 
letterpaper
]{amsart}

\usepackage[margin=1in]{geometry}

\usepackage{amsmath}
\usepackage{graphicx}

\title[The radiation field in $3+1$-dimensions]{An explicit
  description of the radiation field in $3+1$-dimensions} 
\author{Dean Baskin}
\address{Department of Mathematics, Texas A\&M University, College
  Station, TX 77843}
\email{dbaskin@math.tamu.edu}
\date{11 April 2016}
\thanks{The author is grateful to Semyon Dyatlov for an observation
  leading to this note.  This research was conducted with the 
  support of NSF grant DMS-1500646}

\newcommand{\lap}{\Delta}
\newcommand{\sphere}{\mathbb{S}}
\newcommand{\reals}{\mathbb{R}}
\newcommand{\scri}{\mathcal{I}}
\DeclareMathOperator{\im}{Im}

\newcommand{\pd}[1][]{\partial_{#1}}
\newcommand{\vect}[1]{\mathbf{#1}}

\newtheorem{theorem}{Theorem}
\newtheorem{prop}[theorem]{Proposition}
\newtheorem{lemma}[theorem]{Lemma}
\newtheorem{question}[theorem]{Problem}

\begin{document}

\begin{abstract}
  In previous work with A.~Vasy and J.~Wunsch, the author established an
  asymptotic expansion for the radiation field on asymptotically
  Minkowski spacetimes and showed that the exponents seen in the
  expansion are given by the poles of a meromorphic family of
  operators on the spacetime's ``boundary at infinity''.  This note
  provides an explicit accounting of these poles when the spacetime is
  $3+1$-dimensional Minkowski space.  We conclude by stating the
  ``resonant states'' for the first few resonances and then posing a
  combinatorial problem.
\end{abstract}

\maketitle

\section{Introduction}
\label{sec:introduction}

For a forward solution $u$ of the inhomogeneous wave equation on
Minkowski space,
\begin{equation*}
    \Box u = f \in C^{\infty}_{c}(\reals^{3}\times \reals),
\end{equation*}
(or, equivalently, a solution $u$ of the homogeneous wave equation
with compactly supported initial data), the \emph{Friedlander
  radiation field} of $u$ encodes the behavior of $u$ near null
infinity.  With A.~Vasy and J.~Wunsch~\cite{BVW1, BVW2}, the author
established an asymptotic expansion of the radiation field on a class
of asymptotically Minkowski spacetimes and showed that the exponents
of the expansion were given by the poles of a meromorphic family of
operators (called $P_{\sigma}^{-1}$ in those papers) on the boundary
at infinity.  The purpose of this note is to identify explicitly these
poles in the setting of $(3+1)$-dimensional Minkowski space.  In
particular, we prove the following theorem:

\begin{theorem}
  \label{thm:main-thm}
  The poles of $P_{\sigma}^{-1}$ are simple and located at $-\imath (k+1)$ for $k=0,
  1, 2, \dots$.  The rank of the polar part of $P_{\sigma}^{-1}$ at
  $\sigma = -\imath (k+1)$ is $\sum_{j=0}^{k}\dim (E_{j}) = (k+1)^{2}$, where
  $E_{j}$ is the eigenspace of $\lap_{\sphere^{2}}$ with eigenvalue
  $j$.  
\end{theorem}

The rest of the introduction is devoted to explaining and motivating
Theorem~\ref{thm:main-thm}.  

Suppose that $u$ is the solution of the homogeneous wave equation (or,
equivalently, a forward solution of the inhomogeneous wave equation) on
$3+1$-dimensional Minkowski space:
\begin{equation*}
  \begin{aligned}
    \Box u &= \pd[t]^{2} - \lap u = 0 \quad \text{in }\reals\times
    \reals^{3}\\
    (u, \pd[t]u)|_{t=0}&= (\phi, \psi) \in C^{\infty}_{c}(\reals^{3})
    \times C^{\infty}_{c}(\reals^{3})
  \end{aligned}
\end{equation*}
We now introduce polar coordinates $(r,\omega)$ in the spatial
variables as well as a ``lapse'' parameter $s = t-r$ and define an
auxiliary function
\begin{equation*}
  v(r, s, \omega) = r^{-1}u(s+r, r\omega).
\end{equation*}
Friedlander~\cite{Friedlander} observed that the function $v$ is
smooth in $\rho = r^{-1}$ and so can be extended to $\rho = 0$.  The
\emph{Friedlander radiation field} of $u$ is then given by
\begin{equation*}
  \mathcal{R}_{+}[u](s,\omega) = \lim_{r\to\infty}\pd[s]v(r, s, \omega).
\end{equation*}
In Minkowski spaces (and other static spacetimes), the radiation field
has a number of desirable properties: Not only is it a unitary translation
representation of the wave group, it can also be thought of as
generalizing the Radon transform of the initial data.  (Indeed, in
$\reals\times \reals^{3}$, if the initial data are $(0, \psi)$, then
the radiation field is just the Radon transform of $\psi$.)

In previous work~\cite{BVW1,BVW2}, the author and collaborators showed
that the radiation field admits an asymptotic expansion for suitably
nice data.  Indeed, on a class of asymptotically Minkowski spacetimes,
the radiation field exists and admits an asymptotic expansion in
powers of $s^{-1}$.  The exponents seen arise as the poles of a
meromorphic family of Fredholm operators, denoted $P_{\sigma}^{-1}$ on
the ``boundary at infinity''.

We let $M$ denote the radial compactification of Minkowski space with
a defining function $\rho$ for the boundary.  In other words, we can
consider $\reals\times \reals^{3}$ as the interior of a compact
manifold with boundary via the coordinate change
\begin{equation*}
    t = \frac{1}{\rho}\cos \theta , \quad
    x = \frac{1}{\rho}\omega_{j}\sin\theta,
\end{equation*}
where $\omega_{j} \in \sphere^{2}$ and $\theta \in \sphere^{1}$.  We
denote by $C_{\pm}$ (depending on the sign of $t$) the regions of the
boundary sphere $X\cong \sphere^{3}$ corresponding to where $|t|\gg
|x|$, while we denote by $C_{0}$ the region of $X$ where $|t|\ll
|x|$.  These regions of $X$ naturally inherit conformal families of
metrics; on Minkowski space, $C_{\pm}$ are naturally conformal to
$\mathbb{H}^{3}$ and $C_{0}$ is naturally conformal to
$2+1$-dimensional de Sitter space.

It is the region where $|t|\sim |x|$ that is of the most interest; we
denote these regions by $S_{\pm}$ depending on the sign of $t$.  By
blowing up (in the algebro-geometric sense) the submanifolds $S_{\pm}$
in $M$, we obtain a manifold with corners that has two new boundary
faces, denoted $\scri^{\pm}$ and corresponding to past and future null
infinity.  Figure~\ref{fig:blowup} provides a schematic view of this
blow-up.  The radiation field $\mathcal{R}_{+}[u]$ is then the
rescaled restriction of $u$ to $\scri^{+}$.

\begin{figure}[htp]
  \centering
  \includegraphics{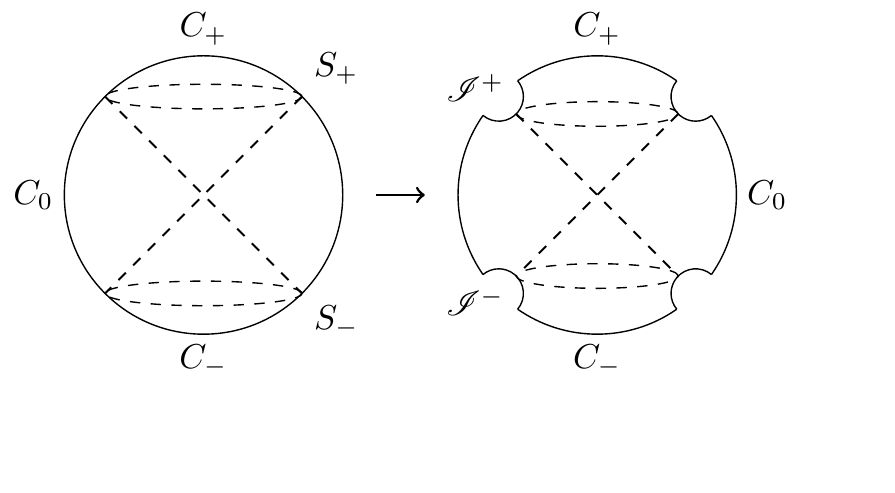}
  \caption{A schematic view of the blow-up.  The lapse function $s$
    increases along $\scri^+$ towards $C_+$.  In the typical Penrose
    diagram of Minkowski space, $C_\pm$ are collapsed to
    $i_\pm$ and $C_0$ is collapsed to $i_0$.}
  \label{fig:blowup}
\end{figure}

Conjugating $\Box$ by $\rho$, multiplying by $\rho^{-2}$, and then
taking the Mellin transform in $\rho$ yields a family of operators
$P_{\sigma}$ on the boundary sphere $X \cong \sphere^{3}$.  This has
the effect of replacing all factors of $\rho\pd[\rho]$ in
$\rho^{-3}\Box \rho$ by $\imath \sigma$.  Although $P_{\sigma}$ is
semiclassically hyperbolic (because $\Box$ is hyperbolic), on
$C_{\pm}$ it is classically elliptic (indeed, it can be conjugated to
the spectral family for the Laplacian on hyperbolic space).
On $C_{0}$, $P_{\sigma}$ is hyperbolic and can be conjugated to a
Klein-Gordon equation on de Sitter space, while at $S_{\pm}$ it is degenerate.

The Hamilton vector field of the symbol of $P_{\sigma}$ is radial at
the conormal bundle of $S_{\pm}$ and so techniques dating back to
Melrose~\cite{Melrose:1994} and refined by
Vasy~\cite{Vasy:kerr-de-sitter} provide a blueprint for establishing
propagation estimates there.

The operator family $P_{\sigma}$ is not Fredholm on standard Sobolev
spaces, but it is when considered on variable-order Sobolev spaces whose
regularity lies below some threshold at $S_{-}$ (depending on the
imaginary part of $\sigma$) and is larger than a similar threshold at
$S_{+}$ (again, depending on the imaginary part of $\sigma$).  As
$P_{\sigma}$ is then invertible on these spaces for very large $\im
\sigma$, we may invert to obtain a meromorphic family of Fredholm
operators $P_{\sigma}^{-1}$.  Because all light rays in Minkowski
space escape to infinity, $P_{\sigma}^{-1}$ has only finitely many
poles in any horizontal strip in $\mathbb{C}$.  The main result of
both previous papers~\cite{BVW1,BVW2} is that the radiation field for
a forward solution has an asymptotic expansion whose exponents are
these poles, which are identified as the resonances of the
asymptotically hyperbolic operator at $C_{+}$.  

For $3+1$-dimensional Minkowski space, however, this asymptotically
hyperbolic operator is the spectral family of the Laplacian on
$\mathbb{H}^{3}$, \emph{which has no resonances}.  In this case, the
resonant states associated to the poles of $P_{\sigma}^{-1}$ must be
supported in $S_{+}$ (rather than $\overline{C_{+}}$).
Theorem~\ref{thm:main-thm} describes the locations of these poles and
the dimension of the corresponding nullspace of $P_{\sigma}$.

The proof of Theorem~\ref{thm:main-thm} proceeds in several steps.  We
first recall from~\cite{BVW1} that any resonant state must be
supported at the intersection of the light cones and the boundary at
infinity because odd-dimensional hyperbolic space has no resonances.
This reduces the problem of finding the poles of $P_{\sigma}^{-1}$
(and the corresponding resonant states) to understanding when
$P_{\sigma}$ has nullspace consisting of a distribution of the form
\begin{equation*}
  \sum_{k=0}^{M}a_{k}\delta^{(k)}(v)\otimes \phi_{\lambda},
\end{equation*}
where $\phi_{\lambda}$ is a spherical harmonic with eigenvalue
$\lambda$.  This immediately implies that the poles of
$P_{\sigma}^{-1}$ are contained in the negative imaginary integers and
reduces the problem of finding the null spaces of an explicit family
of matrices.  For $\sigma = - \imath (M+1)$, $P_{\sigma}$ preserves
the family of such distributions and so the problem of finding the
null space of $P_{\sigma}$ reduces to a linear algebra problem.  We
write down the matrix representing $P_{\sigma}$ and compute its
determinant explicitly.  This shows that the matrix has
one-dimensional null space precisely when $\lambda = k(k+1)$ for
$k = 0, 1, \dots, M$ and hence that $P_{-\imath (M+1)}$ has null space
of dimension
\begin{equation*}
  \sum_{k=0}^{M} \dim (E_{k}),
\end{equation*}
where $E_{k}$ is the space of spherical harmonics with eigenvalue $k(k+1)$.

Unfortunately, finding the elements of the null space explicitly is
sufficiently complicated that we were unable to solve it here by
purely combinatorial means.  It is perhaps surprising how difficult it
is to find an explicit expression for elements of the nullspace of
$P_{-\imath (k+1)}$.  However, given the connection of the radiation
field with the Radon transform, such an expression would provide an
explicit formula for the Radon transform in terms of spherical
harmonics.  Such formulas exist but are similarly complicated (and
have representation-theoretic underpinnings).

In Section~\ref{sec:geometry} we describe some of the geometry of the
radial compactification of Minkowski space, and then in
Section~\ref{sec:operators} we define the operator $P_{\sigma}$ and
recall some of its properties.  We also introduce a convenient
coordinate system that simplifies the linear algebra in the following
section.  Section~\ref{sec:poles-p_sigma-1}
recasts the problem of describing the poles and corresponding resonant
states in terms of linear algebra and finishes the proof of
Theorem~\ref{thm:main-thm}.  Included in
Section~\ref{sec:poles-p_sigma-1} is the exact form of the resonant
states corresponding to the first few poles of $P_{\sigma}^{-1}$.
Finally, in Section~\ref{sec:comb-probl} we conclude by appealing for
a combinatorial expression for the elements of the null space of
$P_{\sigma}$ in this context.

\section{Geometry}
\label{sec:geometry}

The radiation field is the rescaled restriction of a solution $u$ of
the wave equation to null infinity.  In this section we describe a
compactification of Minkowski space on which this statement is a
natural one.  

We begin by introducing coordinates on Minkowski space given by
\begin{align*}
  t &= \frac{1}{\rho} \cos \theta \\
  x &= \frac{1}{\rho}\omega_{j} \sin\theta
\end{align*}
where $\omega_{j} \in \sphere^{2}$.  In terms of these coordinates,
the metric on Minkowski space is given by
\begin{equation*}
  g := -dt^{2} + \sum_{j=1}^{3}dx_{j}^{2} = -\cos 2\theta
  \frac{d\rho^{2}}{\rho^{4}} - 4 \sin\theta\cos\theta
  \frac{d\theta}{\rho}\frac{d\rho}{\rho^{2}} + \cos 2\theta
  \frac{d\theta^{2}}{\rho^{2}} + \sin^{2}\theta \frac{d\omega^{2}}{\rho^{2}}
\end{equation*}
We now replace the coordinate $\theta$ by $v = \cos 2\theta$ to obtain
\begin{equation*}
  g = - v \frac{d\rho^{2}}{\rho^{4}} +
  \frac{dv}{\rho}\frac{d\rho}{\rho^{2}} +
  \frac{v}{4(1-v^{2})}\frac{dv^{2}}{\rho^{2}} + \frac{1-v}{2}\frac{d\omega^{2}}{\rho^{2}}
\end{equation*}
The inverse metric (in coordinates $(\rho, v, \omega)$) is then given
by
\begin{equation*}
  g^{-1} \to
  \begin{pmatrix}
    -v \rho^{4} & 2(1-v^{2})\rho^{3} & 0  \\
    2(1-v^{2})\rho^{3} & 4v(1-v^{2})\rho^{2} & 0 \\
    0 & 0 & \frac{2\rho^{2}}{1-v} h^{-1}
  \end{pmatrix},
\end{equation*}
where $h$ is the standard (round) metric on $\sphere^{2}$.

This radial compactification of Minkowski space has two distinguished
submanifolds where $\rho = 0$ and $v=0$, which we call $S_{\pm}$
($S_{\pm}$ are distinguished by the sign of $t$ -- $S_{+}$ is the set
in the future where $\rho = 0$ and $v=0$, while $S_{-}$ is the
corresponding set in the past.

\section{The operators}
\label{sec:operators}

The central object of study is the operator $L$, given by
\begin{equation*}
  L = \rho^{-3}\Box_{g} \rho.
\end{equation*}
Here the conjugation by $\rho$ should be thought of as accounting for
the standard decay for solutions of the wave equation, while the
prefactor of $\rho^{-2}$ turns a ``scattering operator'' in the sense
of Melrose~\cite{Melrose:1994} into a ``b-operator''~\cite{Melrose:APS}.

We record here the precise form of $L$:
\begin{align*}
  L&= v (\rho\pd[\rho])^{2} + (2 + 4v)\rho\pd[\rho] -
     4(1-v^{2})\rho\pd[\rho]\pd[v] - 4v(1-v^{2})\pd[v]^{2} -
     4(1-v-3v^{2})\pd[v] - \frac{2}{1-v}\lap_{\omega} + (2+3v)
\end{align*}

The operator $P_{\sigma}$ is the reduced normal operator
$\hat{N}(L)(\sigma)$, which effectively replaces $\rho\pd[\rho]$ by
$\imath\sigma$ and is obtained by conjugating $L$ by the Mellin
transform in $\rho$:
\begin{align*}
  P_{\sigma} &= -v\sigma^{2} + (2+4v)\imath\sigma - 4\imath\sigma
               (1-v^{2})\pd[v] - 4v(1-v^{2})\pd[v]^{2} -
               4(1-v-3v^{2})\pd[v] - \frac{2}{1-v}\lap_{\omega} + (2+3v)
\end{align*}

Although the expression above for $P_{\sigma}$ is useful for the
global problem of identifying the Fredholm properties of $P_{\sigma}$,
for our explicit computation it is more convenient to work with a
different coordinate system valid near $S_{+}$.  For the remainder of
this note, we instead take
\begin{equation*}
  \rho = \frac{1}{t + r} , \quad  v= \frac{t-r}{t+r}.
\end{equation*}
In these coordinates, we may write
\begin{equation*}
  \begin{aligned}
    \Box &= \pd[t]^{2} - \pd[r]^{2} - \frac{2}{r}\pd[r] -
    \frac{1}{r^{2}}\lap_{\omega} \\
    &= 4\rho^{2} \left[ -\rho \pd[\rho] \pd[v] - \pd[v] - v\pd[v]^{2}
      + \frac{1}{1-v}(\rho \pd[\rho] + (1 + v) \pd[v]) - \frac{1}{(1-v)^{2}}\lap_{\omega}\right]
  \end{aligned}
\end{equation*}
We then have
\begin{equation*}
  L = \rho^{-3}\Box \rho = 4 \left[ - \rho \pd[\rho] \pd[v] - 2 \pd[v]
    - v \pd[v^{2}] + \frac{1}{1-v} (\rho \pd[\rho] + 1 + (1+v)\pd[v])
    - \frac{1}{(1-v)^{2}}\lap_{\omega}\right]
\end{equation*}
and
\begin{equation*}
  P_{\sigma} = 4 \left[ -(\imath \sigma + 2) \pd[v] - v\pd[v]^{2} +
    \frac{\imath \sigma + 1}{1-v} + \frac{1+v}{1-v}\pd[v] - \frac{1}{(1-v)^{2}}\lap_{\omega}\right].
\end{equation*}
We may then multiply $P_{\sigma}$ by $(1-v)^{2}/4$ and group the terms
with the same degree of homogeneity:
\begin{equation*}
  \begin{aligned}
    \frac{(1-v)^{2}}{4}P_{\sigma} &= \left[ -(\imath \sigma + 1)
      \pd[v] - v\pd[v]^{2} \right] \\
    &\quad + \left[ 2(\imath \sigma + 2)v\pd[v] + 2 v^{2}\pd[v]^{2} +
      (\imath \sigma + 1) - \lap \omega \right] \\
    &\quad + \left[ - (\imath \sigma + 3)v^{2} \pd[v] -
      v^{3}\pd[v]^{2} - (\imath \sigma  + 1)v\right]
  \end{aligned}
\end{equation*}

\section{The poles of $P_\sigma^{-1}$}
\label{sec:poles-p_sigma-1}

At each pole of $P_{\sigma}^{-1}$, the residue can be identified with
an operator whose image is a ``resonant state''.  As $P_{\sigma}$ is
not self-adjoint, the residue operators do not project onto these
states, but we abuse terminology by calling them resonant states
anyway.  We know from the results of~\cite{BVW1} that if $f$ is
supported away from $\overline{C_{-}}$, then $P_{\sigma}^{-1}f$ is
supported in $\overline{C_{+}}$.  Moreover, if $P_{\sigma}^{-1}f$ is
not supported in $S_{+}$, then the pole (and corresponding state) can
be identified with a resonance of a Laplace-like operator on $C_{+}$.
In $n+1$-dimensional Minkowski space, this operator is the Laplacian
on $\mathbb{H}^{n}$.  

Because $\mathbb{H}^{3}$ has no resonances, we can conclude that all
resonant states of $P_{\sigma}^{-1}$ on $3+1$-dimensional Minkowski
space must be supported in $S_{+}$.  The resonant states must
therefore be sums of the following form:
\begin{equation}
  \label{eq:test-distributions}
  \sum_{k=0}^{M}\alpha_{k}\delta^{(k)}(v) \otimes \phi_{\lambda}(\omega),
\end{equation}
where $\phi_{\lambda}$ is a spherical harmonic with eigenvalue
$\lambda$.  

In the rest of this section, we prove Theorem~\ref{thm:main-thm} in
several steps.  We first compute the action of $P_{\sigma}$ on such a
sum, which shows that $P_{\sigma}$ has no null space unless
$\sigma = -\imath (M+1)$ for $M = 0 , 1, \dots$.  For such a $\sigma$,
we then interpret $\frac{(1-v)^{2}}{4}P_{\sigma}$ (which has the same
null space as $P_{\sigma}$ when acting on such distributions) as a
family of matrices depending on $\lambda$.  We compute this
determinant and show that the matrix has a $1$-dimensional null space
exactly when $\lambda = k(k+1)$ for $k = 0 ,1, \dots, M$.  Appealing
to the well-known dimension of the space of spherical harmonics with
eigenvalue $k(k+1)$ then completes the proof of
Theorem~\ref{thm:main-thm}.  

We now record the action of $\frac{(1-v)^{2}}{4}P_{\sigma}$ on
distributions of the form above~\eqref{eq:test-distributions}.  We
rely on the following well-known fact:\footnote{This formula is \emph{a
  priori} valid only for $r \leq k$, but if we interpret
  $(k-r)! = \Gamma ( k -r + 1)$, then the denominator is infinite for
  $r > k$ and so the right-hand side is zero there.}
\begin{equation*}
  v^{r}\delta^{k}(v) = \frac{k!}{(k-r)!}\delta^{(k-r)}(v)
\end{equation*}
We then have the following:
\begin{equation*}
  \begin{aligned}
    \frac{(1-v)^{2}}{4}P_{\sigma}( \delta^{(k)}(v) \otimes
    \phi_{\lambda}(\omega)) &= \left[ (-\imath \sigma + 1) +
      (k+2)\right] \delta ^{(k+1)}\otimes \phi_{\lambda} \\
    &\quad + \left[ -2(\imath \sigma + 2)(k+1) + 2 (k+1)(k+2) +
      (\imath \sigma + 1) + \lambda\right]\delta ^{(k)}\otimes
    \phi_{\lambda} \\
    &\quad + \left[ -(\imath\sigma + 3)k(k+1) + k(k+1)(k+2) + (\imath
      \sigma + 1)k\right] \delta^{(k-1)}\otimes \phi_{\lambda}
  \end{aligned}
\end{equation*}
In particular, for a sum of the form~\eqref{eq:test-distributions} to
lie in the null space of $P_{\sigma}$, the leading term must vanish
and so $M+1 - \imath \sigma = 0$, i.e., $\imath \sigma = M+1$.  We may
then take $\sigma = -\imath (M+1)$, apply $P_{\sigma}$ to such a
sum, and rearrange the terms to find the following:
\begin{equation}
  \label{eq:P-sigma-on-sum}
  \begin{aligned}
    \frac{(1-v)^{2}}{4}P_{-\imath (M+1)}\left(
      \sum_{k=0}^{M}\alpha_{k}\delta^{(k)}(v) \otimes
      \phi_{\lambda}(\omega\right) &= \sum _{k=0}^{M} \left[
      (k-1-M)\alpha_{k-1} \right.\\
      &\quad + (\lambda + M + 2 + 2(k+1)(k-M-1))\alpha_{k} \\
      &\quad \left. + (k+1)(M+2
      + (k+2)(k-M-1))\alpha_{k+1} \right]\delta^{(k)}(v)\otimes
    \phi_{\lambda}(\omega)
  \end{aligned}
\end{equation}

We have now reduced the problem to finding a vector of coefficients
$\alpha_{k}$ so that the sum~\eqref{eq:P-sigma-on-sum} vanishes.  This
is equivalent to finding the null space of a matrix $\lambda I -
A_{M}$, where $A_{M}$ is the tridiagonal $(M+1)\times (M+1)$-matrix
with the following entries:
\begin{equation*}
  \begin{aligned}
    a_{k,k} &= 2(k+1)(M+1-k) - M - 2\\
    a_{k-1,k} &= k( (k+1)(M+2 - k) - M - 2) \\
    a_{k,k-1} &= M + 1 - k
  \end{aligned}
\end{equation*}
Here all indices ($k$ and $k-1$) should be interpreted as taking the values
$0, 1 , \dots, M$.

The rest of the proof of Theorem~\ref{thm:main-thm} then follows from
the following proposition:
\begin{prop}
  \label{prop:det-calc}
  The matrix $A_{M}$ has simple eigenvalues $k(k+1)$ for $k = 0, 1,
  \dots, M$.  In particular, we have that
  \begin{equation*}
    \det (\lambda I - A_{M} ) = p_{M}(\lambda), 
  \end{equation*}
  where
  \begin{equation*}
    p_{k}(\lambda) = \prod_{j=0}^{k}(\lambda - j(j+1)).
  \end{equation*}
\end{prop}

Proposition~\ref{prop:det-calc} follows immediately by taking $k = M$
in the following lemma:
\begin{lemma}
  \label{lemma:full-det-calc}
  Let $d_{k}$ be the determinant of the $(k+1)\times (k+1)$-minor of
  $\lambda I  - A_{M}$ consisting of the first $k+1$ columns and rows
  of the matrix (i.e., the columns and rows labeled $0, 1, \dots,
  k$).  If $p_{k}(\lambda)$ is as in Proposition~\ref{prop:det-calc}, then
  \begin{equation*}
    d_{k} = \sum_{\ell = 0}^{k+1} c_{k.\ell} \left( \prod
      _{j=1}^{\ell} (M - k + \ell - j)\right) p_{k-\ell}(\lambda),
  \end{equation*}
  where
  \begin{equation*}
    c_{k,\ell}= \frac{(-1)^{\ell}}{\ell !} \left( \frac{(k+1)!}{(k+1-\ell)!}\right)^{2},
  \end{equation*}
  and we interpret $p_{-1}(\lambda) = 1$.
\end{lemma}
Observe that all terms containing $p_{k-\ell}(\lambda)$ with
$\ell \geq 1$ in the expression for $d_{k}$ in
Lemma~\ref{lemma:full-det-calc} are multiplied by a factor $(M-k)$ and
hence vanish when $k=M$, leaving only the term $c_{M, 0}
p_{M}(\lambda) = p_{M}(\lambda)$.

\begin{proof}[Proof of Lemma~\ref{lemma:full-det-calc}]
  The matrix $A_{M}$ is tridiagonal, so $d_{k}$ can be computed recursively:
  \begin{equation*}
    d_{k} = (\lambda - a_{k,k})d_{k-1} + a_{k,k-1}a_{k-1,k}d_{k-2}.
  \end{equation*}
  We therefore proceed by induction, interpreting $d_{-1} = 1$ so that the
  lemma holds for $k = -1$.  In particular, we have $c_{-1, 0}= 1$ and
  $c_{-1,\ell}=0$ for $\ell \geq 1$.

  In computing below, we use the following relationship between the $p_{k}$:
  \begin{equation*}
    \lambda p_{k-1-\ell}(\lambda) = p_{k-\ell}(\lambda) +
    (k-\ell)(k-\ell + 1)p_{k-1-\ell}(\lambda)
  \end{equation*}

  We now compute the two terms in the recursive expression for
  $d_{k}$.  We first have the following:
  \begin{equation*}
    \begin{aligned}
      (\lambda - a_{k,k})d_{k-1} &= \left( \lambda + M + 2 -
        2(k+1)(M+1-k)\right) \sum_{\ell = 0}^{k}c_{k-1,\ell} \left(
        \prod _{j=1}^{\ell} (M- k + 1 + \ell - j)\right)
      p_{k-1-\ell}(\lambda) \\
      &= \sum_{\ell = 0}^{k} c_{k-1,\ell} \left(
        \prod_{j=0}^{\ell-1}(M-k+\ell - j)\right) p_{k-\ell}(\lambda)
      \\
      &\quad + \sum_{\ell = 1}^{k+1} c_{k-1,\ell-1}\left(
        \prod_{j=1}^{\ell-1}(M-k+\ell-j)\right) \left[ (k-\ell + 1)
        (k-\ell + 2) + M \right.\\
        &\quad\quad\quad\left.+ 2 - 2(k+1)(M+1-k)\right]p_{k-\ell}(\lambda)
    \end{aligned}
  \end{equation*}
  The second term is given by
  \begin{equation*}
    \begin{aligned}
      a_{k,k-1}a_{k-1,k}d_{k-2} = \sum_{\ell =
        2}^{k+1}c_{k-2,\ell-2}\left(
        \prod_{j=1}^{\ell-1}(M-k+\ell-j)\right)k(M+2+(k+1)(k-M-2))p_{k-\ell}(\lambda)
    \end{aligned}
  \end{equation*}
  Adding the two terms and equating coefficients with the desired
  expression for $d_{k}$, we find that
  \begin{equation*}
    \begin{aligned}
      (M-k)c_{k,\ell} &= (M-k+\ell)c_{k-1,\ell} \\
      &\quad + \left[ (k-\ell+1)(k-\ell +2) + 2(k+1)(k-M-1) + M+2
      \right]c_{k-1,\ell-1} \\ 
      &\quad + k \left[ M + 2 + (k+1)(k-M-2)\right]c_{k-2,\ell-2}.
    \end{aligned}
  \end{equation*}
  We rewrite this equation suggestively:
  \begin{equation}
    \label{eq:recursive-ck}
    \begin{aligned}
      (M-k)c_{k,\ell} &= (M-k)\left( c_{k-1,\ell} -
        (2k+1)c_{k-1,\ell-1} - k^{2}c_{k-2,\ell-2}\right) \\
      &\quad + \ell c_{k-1,\ell} + \left( (k-\ell+1)(k-\ell+2) - k
      \right)c_{k-1,\ell-1} - k^{2}c_{k-2,\ell-2}
    \end{aligned}
  \end{equation}
  
  To prove the lemma, it therefore suffices to show that $c_{k,\ell}$
  is an integer.  We prove this fact by induction.  We have already
  seen that $c_{-1, 0}=1$ and $c_{-1,\ell}=0$ for $\ell \geq 1$.  By
  the induction hypothesis, we assume that
  \begin{equation*}
    c_{k',\ell'} = \frac{(-1)^{\ell'}}{(\ell')!}\left( \frac{(k'+1)!}{(k'+1-\ell')!}\right)^{2}
  \end{equation*}
  for all $(k',\ell') < (k,\ell)$, where we define $(a',b') < (a,b)$ if either
  \begin{itemize}
  \item $b' < b$, or 
  \item $b'=b$ and $a' < a$.
  \end{itemize}
  
  We turn first to the second line of
  equation~\eqref{eq:recursive-ck}.  By the induction hypothesis,
  \begin{equation*}
    \begin{aligned}
      &\ell c_{k-1,\ell} + \left( (k-\ell+1)(k-\ell+2) - k
      \right)c_{k-1,\ell-1} - k^{2}c_{k-2,\ell-2} \\
      &\quad = \frac{(-1)^{\ell}}{(\ell-1)!}\left(
        \frac{k!}{(k-\ell+1)!}\right)^{2}\left( (k-\ell+1)^{2} -
        (k-\ell+1)^{2} - (k-\ell+1) + k - (\ell - 1)\right) = 0
    \end{aligned}
  \end{equation*}
  Equation~\eqref{eq:recursive-ck} and the induction hypothesis then imply that 
  \begin{equation*}
    \begin{aligned}
      c_{k,\ell} &= c_{k-1,\ell} - (2k+1)c_{k-1,\ell-1} -
      k^{2}c_{k-2,\ell-2} \\
      &= \frac{(-1)^{\ell}}{\ell !}\left(
        \frac{k!}{(k-\ell+1)!}\right)^{2} \left( (k-\ell+1)^{2} +
        (2k+1)\ell - \ell (\ell -1)\right) \\
      &= \frac{(-1)^{\ell}}{\ell !} \left( \frac{(k+1)!}{(k-\ell+1)!}\right)^{2},
    \end{aligned}
  \end{equation*}
  finishing the proof of the lemma.
\end{proof}

\subsection{The first few resonant states}
\label{sec:first-few-poles}

In this section we record the first five sets of eigenvectors of the
matrix $A_{M}$.

For $M= 0$, we have that $A_{0} = (0)$, so its only eigenvalue is $0$
with eigenvector $(1)$.

For $M=1$, we have that
\begin{equation*}
  A_{1} =
  \begin{pmatrix}
    1 & 1 \\ 1 & 1
  \end{pmatrix}
,
\end{equation*}
so that the eigenvectors are
\begin{equation*}
  \vect{v}_{0} =
  \begin{pmatrix}
    -1 \\ 1
  \end{pmatrix}
, \quad \vect{v}_{2} =
\begin{pmatrix}
  1 \\ 1
\end{pmatrix}
.
\end{equation*}

For $M=2$, we have
\begin{equation*}
  A_{2} =
  \begin{pmatrix}
    2 & 2 & 0 \\ 2 & 4 & 4 \\ 0 & 1 & 2
  \end{pmatrix}
,
\end{equation*}
with eigenvectors
\begin{equation*}
  \vect{v}_{0} =
  \begin{pmatrix}
    2 \\ -2 \\ 1
  \end{pmatrix}
, \quad \vect{v}_{2} =
\begin{pmatrix}
  -2 \\ 0 \\ 1
\end{pmatrix}
, \quad \vect{v}_{6} =
\begin{pmatrix}
  2 \\ 4 \\ 1
\end{pmatrix}
.
\end{equation*}

For $M=3$, we record the matrix
\begin{equation*}
  A_{3} =
  \begin{pmatrix}
    3  & 3 & 0 & 0 \\ 3 & 7 & 8 & 0 \\ 0 & 2 & 7 & 9 \\ 0 & 0 & 1 & 3
  \end{pmatrix}
,
\end{equation*}
so that the eigenvectors are
\begin{equation*}
  \vect{v}_{0} =
  \begin{pmatrix}
    -6 \\ 6 \\ -3 \\ 1 
  \end{pmatrix}
  , \quad \vect{v}_{2} =
  \begin{pmatrix}
    6 \\ -2 \\ -1 \\ 1
  \end{pmatrix}
  , \quad \vect{v}_{6} =
  \begin{pmatrix}
    -6 \\ -6 \\ 3 \\ 1
  \end{pmatrix}
  , \quad \vect{v}_{12} =
  \begin{pmatrix}
    6 \\ 18 \\ 9 \\ 1
  \end{pmatrix}
.
\end{equation*}

Finally, for $M=4$, the matrix is
\begin{equation*}
  A_{4}=
  \begin{pmatrix}
    4 & 4 & 0 & 0 & 0 \\
    4 & 10 & 12 & 0 & 0 \\
    0 & 3 & 12 & 18 & 0 \\
    0  & 0 & 2 & 10 & 16 \\
    0 & 0 & 0 & 1 & 4
  \end{pmatrix}
,
\end{equation*}
so that the eigenvectors are
\begin{equation*}
  \vect{v}_{0} =
  \begin{pmatrix}
    24 \\ -24 \\ 12 \\ -4 \\ 1
  \end{pmatrix}
  , \quad \vect{v}_{2} =
  \begin{pmatrix}
    -24 \\ 12 \\ 0 \\ -2 \\ 1
  \end{pmatrix}
  , \quad \vect{v}_{6} =
  \begin{pmatrix}
    24 \\ 12 \\ -12 \\ 2 \\ 1
  \end{pmatrix}
  , \quad \vect{v}_{12} =
  \begin{pmatrix}
    -24 \\ -48 \\ 0 \\ 8 \\ 1
  \end{pmatrix}
  , \quad \vect{v}_{20} =
  \begin{pmatrix}
    24 \\ 96 \\ 72 \\ 16 \\ 1
  \end{pmatrix}
.
\end{equation*}

\section{A combinatorial problem}
\label{sec:comb-probl}

We conclude this note by posing a combinatorial problem.  In principle
it is possible to determine the resonant states of $P_{\sigma}^{-1}$
by purely combinatorial means.  Specifically, this would be achieved by explicitly
finding the eigenvectors of the matrix $A_{M}$ above.  The computation
above shows that the eigenvalues are $k(k+1)$ for $k = 0, \dots , M$.

\begin{question}
  \label{problem}
  Find a general expression for the eigenvectors of the matrix
  $A_{M}$.
\end{question}

The resolution of this problem would provide an explicit formula for
the resonant states of $P_{\sigma}$ on Minkowski space and thus give an
explicit expression for the radiation field in terms of spherical
harmonics.  Such a formula would then make it feasible to compute the
radiation field explicitly for non-trivial examples.

Moreover, given the connection between the radiation field and the
Radon transform, such a formula should also recover a formula for the
Radon transform in terms of a spherical harmonic decomposition.
Existing formulas typically rely on the Funk--Hecke formula and thus
involve the Legendre polynomials.  We therefore expect that the
general expression for the eigenvectors of $A_{M}$ ought to be
expressible in terms of coefficients of Legendre polynomials.


\begin{thebibliography}{1}

\bibitem{BVW1}
Dean Baskin, Andr{\'a}s Vasy, and Jared Wunsch.
\newblock Asymptotics of radiation fields in asymptotically {M}inkowski space.
\newblock {\em Amer. J. Math.}, 137(5):1293--1364, 2015.

\bibitem{BVW2}
Dean Baskin, Andr{\'a}s Vasy, and Jared Wunsch.
\newblock Asymptotics of scalar waves on long-range asymptotically {M}inkowski
  spaces.
\newblock Preprint, arXiv:1602.04795, 2016.

\bibitem{Friedlander}
F.~G. Friedlander.
\newblock Radiation fields and hyperbolic scattering theory.
\newblock {\em Math. Proc. Cambridge Philos. Soc.}, 88(3):483--515, 1980.

\bibitem{Melrose:APS}
Richard~B. Melrose.
\newblock {\em The {A}tiyah-{P}atodi-{S}inger index theorem}, volume~4 of {\em
  Research Notes in Mathematics}.
\newblock A K Peters, Ltd., Wellesley, MA, 1993.

\bibitem{Melrose:1994}
Richard~B. Melrose.
\newblock Spectral and scattering theory for the {L}aplacian on asymptotically
  {E}uclidian spaces.
\newblock In {\em Spectral and scattering theory ({S}anda, 1992)}, volume 161
  of {\em Lecture Notes in Pure and Appl. Math.}, pages 85--130. Dekker, New
  York, 1994.

\bibitem{Vasy:kerr-de-sitter}
Andr{\'a}s Vasy.
\newblock Microlocal analysis of asymptotically hyperbolic and {K}err-de
  {S}itter spaces (with an appendix by {S}emyon {D}yatlov).
\newblock {\em Invent. Math.}, 194(2):381--513, 2013.

\end{thebibliography}

\end{document}